\documentclass{amsart}
\usepackage[margin=38mm,a4paper]{geometry}
\usepackage{amsmath,amssymb,amsthm,mathtools}
\usepackage{mathrsfs}
\usepackage{microtype}
\usepackage{enumitem}
\usepackage{tikz}
\usetikzlibrary{arrows.meta}
\usepackage{hyperref}

\newtheorem{maintheorem}{Theorem}

\newtheorem{maincorollary}[maintheorem]{Corollary}
\newtheorem{theorem}{Theorem}[section]
\newtheorem{proposition}[theorem]{Proposition}
\newtheorem{lemma}[theorem]{Lemma}
\newtheorem{corollary}[theorem]{Corollary}
\theoremstyle{definition}
\newtheorem{definition}[theorem]{Definition}
\theoremstyle{remark}
\newtheorem{remark}[theorem]{Remark}

\numberwithin{equation}{section}

\newcommand{\ZZ}{\mathbb{Z}}
\newcommand{\NN}{\mathbb{N}}
\newcommand{\FP}[1]{\mathrm{FP}_{#1}}
\newcommand{\Ft}[1]{\mathrm{F}_{#1}}
\newcommand{\orc}{\mathcal{O}}
\newcommand{\Tr}{{\mathscr{T}}}
\newcommand{\rels}{\mathcal{R}}
\newcommand{\ncl}[1]{\langle\!\langle #1\rangle\!\rangle}
\DeclareMathOperator{\girth}{girth}
\DeclareMathOperator{\id}{id}

\begin{document}

\title[Embedding countable groups in groups of type $\mathrm{FP}_n$]{\boldmath Every countable group embeds in a group of type $\mathrm{FP}_n$}
\author{Laurent Bartholdi}
\address{LB: Institut Camille Jordan, Université de Lyon 1 \textnormal{\itshape and} Section de Mathématiques, Université de Genève}
\email{laurent.bartholdi@gmail.com}
\author{Roman Mikhailov}
\address{RM: St Petersburg State University}


\begin{abstract}
For every integer $n\ge2$, we prove that every countable group embeds in a group of type $\FP{n}$. We also construct a group of type $\Ft{n}$ containing a copy of every recursively presented group. Consequently, a finitely generated group embeds in a group of type $\Ft{n}$ if and only if it is recursively presented.
This answers long-open questions from Bestvina's list (Question 8.7), questions by Fournier-Facio and Zaremsky, and confirms a suggestion of Gromov.

An appendix combines the construction with the controlled-filling argument in a subsequently released OpenAI preprint to prove that every countable group embeds in a two-generator group of type $\FP{\infty}$.
\end{abstract}

\maketitle

\section{Introduction}

A group is of type $\FP{n}$ if the trivial module $\ZZ$ has a projective resolution over the integral group ring that is finitely generated in degrees at most $n$, and of type $\Ft{n}$ if it has an Eilenberg--Mac\,Lane space with finite $n$-skeleton. Higman~\cite{Hig} proved that a finitely generated group embeds in a finitely presented group if and only if it is recursively presented, and Leary~\cite[Theorem~1.1]{Lea} proved that every countable group embeds in a group of type $\FP{2}$. We extend both embedding results to arbitrary degree $\ge2$:
\begin{maintheorem}\label{thm:A}
Let $n\ge 2$ be an integer. Every countable group embeds in a group of type $\FP{n}$.
\end{maintheorem}

\begin{maintheorem}\label{thm:B}
Let $n\ge 2$ be an integer. There is a group of type $\Ft{n}$ that contains a copy of every recursively presented group.
\end{maintheorem}

\begin{maincorollary}\label{cor:C}
Let $n\ge2$ be an integer. A countable group embeds in a group of type $\Ft{n}$ if and only if it embeds in a finitely presented group. In particular, a finitely generated group embeds in a group of type $\Ft{n}$ if and only if it is recursively presented.
\end{maincorollary}

Misha Gromov suggested, after a talk of the first-named author at the 2015 meeting ``Growth, symbolic dynamics and combinatorics of words in groups'' in Paris, that perhaps the only obstruction to embedding theorems lies in dimension $2$. Corollary~\ref{cor:C} gives a precise finite-degree version of this suggestion: passing from finite presentability to any fixed higher finiteness degree introduces no additional embedding obstruction.

Bestvina's problem list~\cite{Bestvina} asks, in Question~8.7, whether the Pr\"ufer $p$-group $\ZZ_{p^\infty}$ can be embedded in an $\FP{\infty}$ group, or at least an $\FP{3}$ group; and whether every $\Ft{n}$ group can be embedded in an $\Ft{n+1}$ group, if $n\ge2$. Theorem~\ref{thm:A} settles the second question, while Theorem~\ref{thm:B} settles the third one, both in the affirmative.

Theorem~\ref{thm:B} also settles the $\Ft{3}$ part of \cite[Question~1.3]{FFZ}, and Theorem~\ref{thm:A} the $\FP{3}$ part of \cite[Question~1.4]{FFZ}; both hold in every degree.  The target groups in both theorems can be chosen to have two generators; see Corollary~\ref{cor:two-generators}.

After this manuscript was completed, OpenAI proved that every finitely presented group embeds in a group of type $\Ft{\infty}$~\cite{OAI}, which settles the first part of Bestvina's question. Combining its controlled-filling argument with our construction gives the analogous statement for arbitrary countable groups and type $\FP{\infty}$; see Appendix~\ref{app:infinity}:
\begin{maintheorem}\label{thm:O}
Every countable group embeds in a group of type $\FP{\infty}$.
\end{maintheorem}

\subsection*{The degree-raising criterion}
The theorems of Higman and Leary are both proved with the Higman rope trick~\cite[IV.7.6]{LS}, whose homological version is~\cite[Lemma~2.2]{Lea}. Fournier-Facio and Zaremsky showed that the groups produced by the rope trick have infinite-dimensional third rational homology whenever the embedded group $F/R$ is infinite and $R\ne 1$~\cite[Theorem~B]{FFZ}; in particular such groups are not of type $\FP{3}$. We raise the finiteness degree with ascending HNN extensions instead.

Let $m\ge 1$, let $B$ be a group of type $\FP{m}$, and fix an exact sequence of $\ZZ B$-modules
\begin{equation}\label{eq:resolution}
0\xrightarrow{\phantom{\subseteq}} Z\xrightarrow{\subseteq} F_m\xrightarrow{d_m}F_{m-1}\to\cdots\xrightarrow{d_1}F_0=\ZZ B\xrightarrow{\varepsilon}\ZZ\to 0
\end{equation}
in which every $F_i$ is finitely generated and free. Then $B$ is of type $\FP{m+1}$ if and only if $Z$ is finitely generated~\cite[VIII.4.3]{Bro}. Every endomorphism $\varphi$ of $B$ has chain lifts $\Phi$ to~\eqref{eq:resolution} (Definition~\ref{def:lift}), so each of them maps $Z$ into itself.

\begin{maintheorem}\label{thm:raising}
Let $m\ge 1$, let $B$ be a group of type $\FP{m}$ and let $\varphi\colon B\to B$ be an injective endomorphism. Suppose that some chain lift $\Phi$ of $\varphi$ satisfies $\Phi(Z)\subseteq M$ for a finitely generated $\ZZ B$-submodule $M\subseteq Z$. Then $B$ embeds in the group
\[
E=\langle B,t\colon t^{-1}bt=\varphi(b),\ b\in B\rangle,
\]
and $E$ is of type $\FP{m+1}$.
\end{maintheorem}

The hypothesis does not depend on the choice of the chain lift or of the finite free partial resolution; see Lemma~\ref{lem:comparison} and Remark~\ref{rem:resolution-independence}. It requires neither that $\Phi(Z)$ be finitely generated nor that $M$ be $\Phi$-invariant.

\subsection*{Outline}
A fixed finite system of commuting-copy and diagonal relations produces an injective endomorphism whose action on the top cycle module has image contained in a finitely generated submodule. Theorem~\ref{thm:raising} then raises the finiteness degree by one. We describe the three ingredients of this construction. An \emph{oriented triple system} $\Tr$ on a finite set $V$ is a set of triples $(i,j)\to k$ of distinct elements of $V$, with \emph{inputs} $i,j$ and \emph{output} $k$ (Definition~\ref{def:triples}). For a group $G$, let $G_v$ ($v\in V$) be copies of $G$, write $g_v\in G_v$ for the copy of $g\in G$, and put
\begin{equation}\label{eq:KSigma}
K_\Tr(G)=\bigl(\mathop{\ast}_{v\in V}G_v\bigr)\big/\ncl{\,[g_i,h_j],\ g_k^{-1}g_ig_j\ \colon\ g,h\in G,\ (i,j)\to k\in\Tr\,},
\end{equation}
with canonical homomorphisms $\iota_v\colon G\to K_\Tr(G)$, $g\mapsto g_v$. Thus $K_\Tr(G)$ is generated by the copies $G_v$ subject to the following relations, one family for each triple $(i,j)\to k$: the input copies $G_i$ and $G_j$ commute elementwise, and $g_k=g_ig_j$ for every $g\in G$, so that the output copy is identified with the diagonal in the input copies. For a single triple, $K_\Tr(G)=G\times G$, and $G_i$, $G_j$, $G_k$ are the two factors and the diagonal.

Suppose that $\kappa\colon K_\Tr(B)\to B$ is an embedding, and put $\alpha_v=\kappa\circ\iota_v$. For each triple $(i,j)\to k$, the endomorphisms $\alpha_i$ and $\alpha_j$ have commuting images and $\alpha_k(b)=\alpha_i(b)\alpha_j(b)$ for all $b\in B$. For such endomorphisms, chain lifts satisfy $\Phi_k\equiv\Phi_i+\Phi_j$ on $Z$ modulo a finitely generated submodule (Lemma~\ref{lem:product}). Consequently, if the standard basis vector $e_0$ of $\ZZ^V$ is an integral linear combination of the vectors $e_i+e_j-e_k$, the same combination of the maps $\Phi_v$ shows that $\Phi_0$ maps $Z$ into a finitely generated submodule (Corollary~\ref{cor:row}). Theorem~\ref{thm:raising} then applies to $\alpha_0$, provided $\iota_0$ is injective. We show that every $\iota_v$ is injective, for every group $G$, when the incidence graph of $\Tr$ has girth at least $12$, by realising $K_\Tr(G)$ as the fundamental group of a non-positively curved complex of groups (Lemma~\ref{lem:injective}). A single system $\Tr$ satisfying both conditions serves for every group and every degree; it has $15$ triples and is obtained from the flags of the Fano plane (Proposition~\ref{prop:certificate}). Finally, an embedding $K_\Tr(B)\to B$ exists whenever $B$ is itself finitely generated and recursively presented relative to a fixed oracle, and contains every finitely generated group that is recursively presented relative to that oracle. Both properties survive the ascending HNN extension. A two-generator universal group supplies the initial case $m=1$ (Lemma~\ref{lem:initial}).

In particular, for $n=2$ the argument gives a new proof of Leary's theorem that does not use the construction of~\cite{Lea}; he had asked for a simpler proof of his Theorem~1.1 in~\cite[Section 3]{Lea}.

\subsection*{Mitotic groups}
The relations associated to a single triple in~\eqref{eq:KSigma} are the factor-diagonal relations underlying a mitosis in the sense of Baumslag, Dyer and Heller~\cite[Section~4]{BDH}. A supergroup $M$ of a group $H$ is a \emph{mitosis} of $H$ if it is generated by $H$ and two elements $s,d$ such that $[h',s^{-1}hs]=1$ and $d^{-1}hd=h\cdot s^{-1}hs$ for all $h,h'\in H$.
The present construction replaces the conjugations by $s$ and $d$ by a finite network of endomorphisms satisfying the girth and integrality conditions above. Additivity is used on the top cycle module of a partial resolution (Section~\ref{sec:products}), to raise a finiteness property rather than to annihilate ordinary homology.

\subsection*{Relation to the work of Fournier-Facio and Zaremsky}
Fournier-Facio and Zaremsky proved that if every finitely generated recursively presented group embeds in a recursively presented group of type $\FP{n}$, then every finitely presented group embeds in a group of type $\Ft{n}$~\cite[Theorem~A]{FFZ}. Theorem~\ref{thm:oracle} with the empty oracle supplies this hypothesis, so Theorem~\ref{thm:B} can also be deduced from their result. We give an equally short direct proof in Subsection~\ref{sec:B}.

\subsection*{Statement on AI use}
ChatGPT 6 Astra was used to assist with the proof of the main results and the preparation of an initial draft of the manuscript. Claude Opus 5.5 was used in subsequent editing. The arguments and text have been reviewed, checked, and substantially revised by the authors, who are responsible for the final manuscript.

\subsection*{Formalization}
Theorem~\ref{thm:A} has been formalized in \textsf{Lean/Mathlib} and is available on the Palomar repository as\\
\centerline{\small\url{https://palomar-registry.org/entry?id=PALOMAR-2026-09-30-000032}.}

\section{Preliminaries}\label{sec:prelim}

\subsection{Conventions}
Modules are left modules, and finite generation of a module refers to the indicated group ring. All finiteness properties are taken with integer coefficients; see~\cite[Chapter~VIII]{Bro}. A group of type $\FP{1}$ is finitely generated. Presentations are written $\langle X\colon\rels\rangle$. For an injective endomorphism $\varphi$ of a group $B$, the ascending HNN extension $\langle B,t\colon t^{-1}bt=\varphi(b),\ b\in B\rangle$ contains $B$ by Britton's lemma~\cite[IV.2]{LS}; its \emph{height} is the homomorphism to $\ZZ$ sending $t$ to $1$ and $B$ to $0$.

Let $\theta\colon\Gamma\to B$ be a homomorphism of groups, extended to group rings. An additive map $f\colon P\to P'$ from a $\ZZ\Gamma$-module to a $\ZZ B$-module is \emph{$\theta$-semilinear} if $f(\lambda x)=\theta(\lambda)f(x)$ for all $\lambda\in\ZZ\Gamma$ and $x\in P$. We shall use repeatedly the following consequence of the definition:
\begin{equation}\label{eq:semilinear}
\text{if $P$ is generated by $x_1,\dots,x_q$, then } f(P)\subseteq \ZZ B\, f(x_1)+\dots+\ZZ B\, f(x_q).
\end{equation}

\subsection{Chain lifts}\label{sec:lifts}
Throughout Sections~\ref{sec:prelim}--\ref{sec:HNN}, $m\ge 1$, $B$ is a group of type $\FP{m}$, the sequence~\eqref{eq:resolution} is fixed, and $Z=\ker d_m$. We write $F_*$ for the unaugmented complex $F_m\to\cdots\to F_0$. It satisfies
\[
H_0(F_*)\cong\ZZ\ \text{(via $\varepsilon$)},\qquad H_i(F_*)=0\quad(0<i<m),\qquad H_m(F_*)=Z.
\]
Note that $H_m(F_*)$ is the module $Z$, not the group homology $H_m(B;\ZZ)$.

\begin{definition}\label{def:lift}
Let $\theta$ be an endomorphism of $B$. A \emph{chain lift} of $\theta$ is a family of $\theta$-semilinear maps $f_i\colon F_i\to F_i$, $0\le i\le m$, such that $d_if_i=f_{i-1}d_i$ for $1\le i\le m$ and $\varepsilon f_0=\varepsilon$. Its restriction to $Z$, which it maps into $Z$, is denoted by the same letter.
\end{definition}

A chain lift of $\theta$ is the same as a $\ZZ B$-linear chain map, over the identity of $\ZZ$, from $F_*$ to the complex obtained from $F_*$ by restricting scalars along $\theta$. Since the $F_i$ are free and the augmented complex is exact at $F_0,\dots,F_{m-1}$, such a map is constructed degree by degree on bases, as in the comparison theorem~\cite[I.7]{Bro}. In particular, chain lifts exist. They are unique up to homotopy only below the top degree; the following lemma is the substitute that we need.

\begin{lemma}\label{lem:comparison}
If $f$ and $g$ are chain lifts of the same endomorphism $\theta$ of $B$, then $(f-g)(Z)$ is contained in a finitely generated $\ZZ B$-submodule of $Z$.
\end{lemma}

\begin{proof}
Since $\varepsilon(f_0-g_0)=0$ and $\ker\varepsilon=d_1(F_1)$, choose a $\theta$-semilinear map $\sigma_0\colon F_0\to F_1$ with $f_0-g_0=d_1\sigma_0$. Inductively, for $1\le i<m$, construct $\theta$-semilinear maps $\sigma_i\colon F_i\to F_{i+1}$ satisfying
\[
f_i-g_i=d_{i+1}\sigma_i+\sigma_{i-1}d_i.
\]
Indeed, once the preceding maps have been constructed, the chain-map identities give
\[
d_i(f_i-g_i-\sigma_{i-1}d_i)
=(f_{i-1}-g_{i-1}-d_i\sigma_{i-1})d_i=0.
\]
For $i=1$ the last equality follows from the choice of $\sigma_0$, and for $i>1$ from $d_{i-1}d_i=0$. Exactness at $F_i$ lets us lift the values through $d_{i+1}$ and extend $\theta$-semilinearly.

Now $c=f_m-g_m-\sigma_{m-1}d_m$ satisfies $d_mc=0$, so $c(F_m)\subseteq Z$, and $c$ agrees with $f_m-g_m$ on $Z$. Since $c$ is $\theta$-semilinear and $F_m$ is finitely generated, $c(F_m)$ lies in a finitely generated submodule of $Z$ by~\eqref{eq:semilinear}.
\end{proof}

\begin{remark}\label{rem:resolution-independence}
The hypothesis of Theorem~\ref{thm:raising} is also independent of the chosen finite free partial resolution. Let $F'_*$ be another such resolution through degree $m$, put $Z'=\ker d'_m$, and choose comparison maps $u\colon F_*\to F'_*$ and $v\colon F'_*\to F_*$ lifting $\id_\ZZ$. If $\Phi$ is a chain lift of $\varphi$ with $\Phi(Z)\subseteq M$, then $u\Phi v$ is a chain lift of $\varphi$ on $F'_*$ and maps $Z'$ into the finitely generated submodule $u_m(M)\subseteq Z'$. Lemma~\ref{lem:comparison}, applied to $F'_*$, gives the same conclusion for every chain lift on $F'_*$. Interchanging the two resolutions proves independence.
\end{remark}

\subsection{Presentations relative to an oracle}\label{sec:oracle}
Fix $\orc\subseteq\NN$. All alphabets below are finite or countable, with a fixed effective enumeration, so that finite words have a computable encoding by natural numbers. A set of words is \emph{$\orc$-recursively enumerable} ($\orc$-r.e.) if some Turing machine with oracle $\orc$ enumerates it. A presentation $\langle X\colon\rels\rangle$ is $\orc$-r.e.\ if $\rels$ is, and a countable group is \emph{$\orc$-recursively presented} if it admits such a presentation. For $\orc=\emptyset$ these are the countable recursively presented groups.

A finitely generated $\orc$-recursively presented group has an $\orc$-r.e.\ presentation on any finite generating set, since the relations among finitely many words in the generators of an $\orc$-r.e.\ presentation form an $\orc$-r.e.\ set. Every countable group has an $\orc$-r.e.\ presentation for a suitable $\orc$: choose a countable presentation and let $\orc$ encode its set of relators.

Adding finitely many generators and relators to an $\orc$-r.e.\ presentation yields an $\orc$-r.e.\ presentation. In particular, if $H$ is finitely generated and $\orc$-recursively presented and $\varphi$ is an injective endomorphism of $H$, then $\langle H,t\colon t^{-1}ht=\varphi(h),\ h\in H\rangle$ is $\orc$-recursively presented: it suffices to impose $t^{-1}xt=w_x$ for $x$ in a finite generating set of $H$, where $w_x$ is any word representing $\varphi(x)$.

\section{A finite system of triples}\label{sec:triples}

\begin{definition}\label{def:triples}
  An \emph{oriented triple system} $\Tr$ on a finite set $V$ is a finite set of triples $\ell=(i,j)\to k$ of distinct elements of $V$; we call $i,j$ the \emph{inputs} and $k$ the \emph{output} of $\ell$, and the order of the inputs is immaterial. The elements of $V$ are called \emph{indices}.

  The \emph{incidence graph} $\Gamma(\Tr)$ is the bipartite graph with vertex set $V\sqcup\Tr$ in which $v\in V$ is joined to $\ell\in\Tr$ when $v$ is an entry of $\ell$. The \emph{row} of $\ell=(i,j)\to k$ is $W_\ell=e_i+e_j-e_k\in\ZZ^V$, where $(e_v)_{v\in V}$ is the standard basis, and $W$ is the $\Tr\times V$ integer matrix with rows $W_\ell$.
\end{definition}

We shall use a system $\Tr$ with a distinguished element $0\in V$ satisfying
\begin{align}
&\girth\Gamma(\Tr)\ge 12, \label{eq:girth}\\
&\nu W=e_0\ \text{for some } \nu\in\ZZ^\Tr. \label{eq:row}
\end{align}
Condition~\eqref{eq:row} means that for every abelian group $A$ and every family $(x_v)_{v\in V}$ in $A$,
\[
\textstyle\sum_{\ell=(i,j)\to k}\nu_\ell\,(x_i+x_j-x_k)=x_0 .
\]
Condition~\eqref{eq:girth} is used only in Lemma~\ref{lem:injective}, and condition~\eqref{eq:row} only in Corollary~\ref{cor:row}.

\begin{proposition}\label{prop:certificate}
There is an oriented triple system $\Tr$ with $23$ indices and $15$ triples satisfying~\eqref{eq:girth} and~\eqref{eq:row}, with $\nu\in\{\pm1\}^\Tr$.
\end{proposition}

\begin{proof}
Let the points of the Fano plane be the elements of $\ZZ/7$ and its lines the sets $L_j=\{j,j+1,j+3\}$, $j\in\ZZ/7$. Its flags are the pairs $f_{j,d}=(j+d,L_j)$ with $j\in\ZZ/7$ and $d\in\{0,1,3\}$; the flags through the point $p$ are $f_{p,0}$, $f_{p-1,1}$ and $f_{p-3,3}$. Let $V$ consist of the twenty flags other than $f_{0,3}=(3,L_0)$ and of three further indices $a$, $b$ and $0$, and let $\Tr$ consist of the fifteen triples
\begin{align*}
\pi_p&=(f_{p,0},f_{p-1,1})\to f_{p-3,3}\quad(p\ne3), & \pi_3&=(f_{3,0},f_{2,1})\to b,\\
\lambda_j&=(f_{j,0},f_{j,1})\to f_{j,3}\quad(j\ne0), & \lambda_0&=(f_{0,0},f_{0,1})\to a,\\
\gamma&=(b,0)\to a.
\end{align*}
Thus $\pi_p$ consists of the flags through $p$ and $\lambda_j$ of the flags on $L_j$, except that the flag $(3,L_0)$ is replaced by $b$ in $\pi_3$ and by $a$ in $\lambda_0$.

Put $\nu=1$ on the triples $\pi_p$ and $\gamma$, and $\nu=-1$ on the triples $\lambda_j$. In $\nu W=\sum_pW_{\pi_p}-\sum_jW_{\lambda_j}+W_\gamma$, a flag $f_{j,0}$ or $f_{j,1}$ is an input of one triple $\pi_p$ and of one triple $\lambda_j$, and contributes $1-1=0$; a flag $f_{j,3}$ with $j\ne0$ is the output of one $\pi_p$ and one $\lambda_j$, and contributes $-1+1=0$; the index $a$ contributes $1$ from $\lambda_0$ and $-1$ from $\gamma$, the index $b$ contributes $-1$ from $\pi_3$ and $1$ from $\gamma$, and the index $0$ contributes $1$ from $\gamma$. Hence $\nu W=e_0$, which is~\eqref{eq:row}.

Every index other than $0$ lies in exactly two triples, and $0$ lies in $\gamma$ only. Let $H$ be the graph with vertex set $\Tr$ and, for each index other than $0$, an edge joining the two triples that contain it. Then $\Gamma(\Tr)$ is obtained from $H$ by subdividing every edge and attaching the pendant vertex $0$ to $\gamma$, so its girth is twice that of $H$. The graph $H$ is the incidence graph of the Fano plane, with $\pi_p$ corresponding to $p$ and $\lambda_j$ to $L_j$, in which the edge joining $3$ to $L_0$ is subdivided by the vertex $\gamma$; see Figure~\ref{fig:triple-graph}. The incidence graph of a projective plane is bipartite and contains no cycle of length $4$, since two points lie on a unique line; hence it has girth at least $6$, and subdividing an edge does not shorten any cycle. Therefore $\Gamma(\Tr)$ has girth at least $12$, which is~\eqref{eq:girth}.
\end{proof}

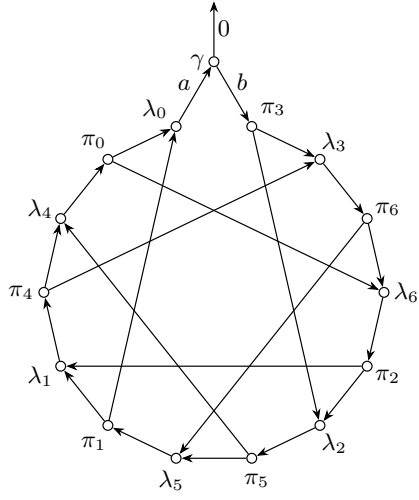
\begin{figure}[tbp]
\centering
\begin{tikzpicture}[line width=.45pt,>={Stealth[length=1.5mm]},
  every node/.style={font=\small,minimum size=0mm,inner sep=1.3pt}]
  \def\R{2.25}
  \foreach \p/\slot in {0/0,3/2,6/4,2/6,5/8,1/10,4/12}{
    \pgfmathsetmacro{\ang}{900/7-\slot*360/14}
    \node[draw,circle] (p\p) at (\ang:\R) {};
    \ifnum\p=3
      \node at (\ang-5:2.55) {$\pi_{\p}$};
    \else
      \node at (\ang:2.55) {$\pi_{\p}$};
    \fi
  }
  \foreach \j/\slot in {0/1,3/3,6/5,2/7,5/9,1/11,4/13}{
    \pgfmathsetmacro{\ang}{900/7-\slot*360/14}
    \node[draw,circle] (l\j) at (\ang:\R) {};
    \ifnum\j=0
      \node at (\ang+5:2.55) {$\lambda_{\j}$};
    \else
      \node at (\ang:2.55) {$\lambda_{\j}$};
    \fi
  }
  \foreach \j in {0,...,6}{
    \draw[<-] (l\j)--(p\j);
    \pgfmathtruncatemacro{\p}{mod(\j+1,7)}
    \draw[<-] (l\j)--(p\p);
  }
  \foreach \j in {1,...,6}{
    \pgfmathtruncatemacro{\p}{mod(\j+3,7)}
    \draw[->] (l\j)--(p\p);
  }
  \node[draw,circle] (gamma) at (0,3.05) {};
  \draw[->] (l0)--node[above left=0pt] {$a$} (gamma);
  \draw[->] (gamma)--node[above right=0pt] {$b$} (p3);
  \draw[->] (gamma)--node[right] {$0$} (0,3.85);
  \node[left=2pt] at (gamma) {$\gamma$};
\end{tikzpicture}
\caption{The graph $H$ of Proposition~\ref{prop:certificate}, with the half-edge $0$ attached. Its vertices are the triples, and its unlabelled edges are the remaining flags of the Fano plane. The edge from $\lambda_0$ to $\pi_3$ is subdivided into $a$ and $b$ by $\gamma$. The orientations are described in Remark~\ref{rem:flow}.}
\label{fig:triple-graph}
\end{figure}

\begin{remark}\label{rem:flow}
The construction is an instance of the following principle. Let $H$ be a finite graph in which every vertex has degree $3$, one of the edges being a half-edge $0$ with a single endpoint, and orient $H$ so that no vertex is a source or a sink and the half-edge $0$ is directed outwards. Taking the triples to be the sets of edges at the vertices, declare the output at each vertex to be the edge whose orientation there differs from that of the other two, and put $\nu=1$ or $\nu=-1$ at vertices with one or two incoming edges respectively. Then the matrix with rows $\nu_\ell W_\ell$ is the oriented incidence matrix of $H$, with entry $1$ at the tail and $-1$ at the head of each edge, so $\nu W$ vanishes on every edge with two endpoints and equals $1$ on the half-edge: this is~\eqref{eq:row}. Dually, a solution of the equations $x_k=x_i+x_j$ is a conserved flow on $H$, and conservation forces the flow $x_0$ through the half-edge to vanish. Condition~\eqref{eq:girth} holds as soon as $H$ has girth at least $6$. In the proof above, $H$ is oriented from points to lines along the flags $f_{j,0}$, $f_{j,1}$ and from lines to points along the flags $f_{j,3}$, the subdivided flag $(3,L_0)$ runs from $L_0$ through $\gamma$ to $3$, and the half-edge $0$ leaves $\gamma$.
\end{remark}

We fix the system $\Tr$ of Proposition~\ref{prop:certificate} for the rest of the paper.

\section{The complex-of-groups construction}\label{sec:complex-groups}

Recall the group $K_\Tr(G)$ from~\eqref{eq:KSigma} and its canonical maps $\iota_v$. For a single triple these maps are injective because $K_\Tr(G)=G\times G$. For a general triple system, the following condition on the incidence graph ensures injectivity.

\begin{lemma}\label{lem:injective}
If $\girth\Gamma(\Tr)\ge 12$, then $\iota_v$ is injective for every $v\in V$ and every group $G$.
\end{lemma}

\begin{proof}
Let $\mathcal{P}$ be the poset with elements $o$, the elements of $V$ and the elements of $\Tr$, ordered by $\ell<v<o$ whenever $v$ is an entry of $\ell$. The monomorphisms of the complex of groups defined below go from larger to smaller elements of $\mathcal{P}$, that is, along $o\to v\to\ell$ and $o\to\ell$. Its geometric realisation $|\mathcal{P}|$ is the cone with apex $o$ on $\Gamma(\Tr)$, whose $2$-simplices are the triangles $\{o,v,\ell\}$ with $v$ an entry of $\ell$. In particular, $|\mathcal{P}|$ is contractible.

Let $G(\mathcal{P})$ be the simple complex of groups over $\mathcal{P}$, in the sense of~\cite[II.12]{BH}, with local groups $G_o=1$, $G_v=G$ for $v\in V$ and $G_\ell=G\times G$ for $\ell\in\Tr$, and with the following monomorphisms: those out of $G_o$ are trivial, and for $\ell=(i,j)\to k$,
\begin{equation}\label{eq:local-maps}
G_i\to G_\ell,\ g\mapsto(g,1);\qquad G_j\to G_\ell,\ g\mapsto(1,g);\qquad G_k\to G_\ell,\ g\mapsto(g,g).
\end{equation}
The compatibility condition on composites is vacuous, because every chain of length two in $\mathcal{P}$ begins at $o$. The direct limit of $G(\mathcal{P})$ is generated by the groups $G_v$ and $G_\ell$ subject to the identifications~\eqref{eq:local-maps}. Since $G\times G$ is generated by $G\times 1$ and $1\times G$ subject only to the relations that these subgroups commute, eliminating the groups $G_\ell$ yields the presentation~\eqref{eq:KSigma}. As $|\mathcal{P}|$ is simply connected, this direct limit is the fundamental group of $G(\mathcal{P})$. It therefore suffices to show that $G(\mathcal{P})$ is developable, since the local groups of a developable complex of groups embed in its fundamental group~\cite[II.12]{BH}.

Make every triangle $\{o,v,\ell\}$ of $|\mathcal{P}|$ a Euclidean triangle with side lengths $|ov|=\sqrt3$, $|v\ell|=1$ and $|o\ell|=2$. Its angles are $\pi/6$ at $o$, $\pi/2$ at $v$ and $\pi/3$ at $\ell$, and edges of the same type have the same length, so this is a piecewise Euclidean structure with a single shape. By~\cite[II.12.28]{BH} (see also~\cite[III.C.4.17]{BH}), $G(\mathcal{P})$ is developable if it is non-positively curved, that is, if its local developments are locally $\mathrm{CAT}(0)$. For a $2$-dimensional piecewise Euclidean complex this is the link condition: every injective loop in the link of a vertex has length at least $2\pi$~\cite[II.5]{BH}. The links at vertices of the local developments are as follows.
\begin{center}\small
\begin{tabular}{lllll}
vertex & local group & link & edge length & girth \\ \hline
$o$ & $1$ & $\Gamma(\Tr)$ & $\pi/6$ & $\ge 12$\\
$v$ & $G$ & $G*\operatorname{Inc}(v)$ & $\pi/2$ & $\ge 4$\\
$\ell$ & $G\times G$ & point--coset incidence graph & $\pi/3$ & $\ge 6$
\end{tabular}
\end{center}
At $o$ the local group is trivial and the link is $\Gamma(\Tr)$. At $v$, put $\operatorname{Inc}(v)=\{\ell\in\Tr\colon v\text{ is an entry of }\ell\}$. The link is the simplicial join $G*\operatorname{Inc}(v)$ of two discrete sets: the set $G_v/G_o\cong G$ of directions towards lifts of $o$, and the set of triples containing $v$. Thus it is a complete bipartite graph.

At $\ell=(i,j)\to k$, let $H_1=G\times1$, $H_2=1\times G$ and let $H_3$ be the diagonal. The link is the bipartite graph with point vertices $z\in G\times G$ and tagged coset vertices $(a,zH_a)$, $a\in\{1,2,3\}$. Each point is joined to the three tagged cosets containing it. Distinct cosets of the same subgroup are disjoint. If $a\ne b$ and two cosets of $H_a$ and $H_b$ meet at $z$, their intersection is $z(H_a\cap H_b)=\{z\}$. Consequently, two distinct tagged cosets have at most one common point. This link has no cycle of length $4$, so its girth is at least $6$. In each case every injective loop has length at least $2\pi$, and the link condition holds.

Thus $G(\mathcal{P})$ is non-positively curved, hence developable, and every $\iota_v$ is injective.
\end{proof}

\begin{lemma}\label{lem:presentations}
The construction $G\mapsto K_\Tr(G)$ is functorial. If $G=\langle X\colon\rels\rangle$ with $X$ finite, then $K_\Tr(G)$ is presented by the generators $x_v$ ($x\in X$, $v\in V$) and the relators
\[
r_v\ (r\in\rels,\ v\in V),\qquad [x_i,y_j],\ \ x_k^{-1}x_ix_j\ \ \bigl(x,y\in X,\ (i,j)\to k\in\Tr\bigr),
\]
where $r_v$ denotes $r$ written in the letters $x_v$. Consequently, $K_\Tr$ preserves finite generation, finite presentability and, for finitely generated groups, $\orc$-recursive presentability.
\end{lemma}

\begin{proof}
Functoriality is clear from~\eqref{eq:KSigma}. In the group given by the displayed presentation, the relators $[x_i,y_j]$ make the images of $G_i$ and $G_j$ commute, so $g\mapsto g_ig_j$ is a homomorphism on $G$. It agrees with $g\mapsto g_k$ on $X$, hence on $G$. So all relators in~\eqref{eq:KSigma} follow from the displayed ones.
\end{proof}

\section{Products of endomorphisms with commuting images}\label{sec:products}

For group homology, the map induced by a pointwise product of homomorphisms with commuting images differs from the sum of the induced maps by mixed terms coming from the Künneth formula; Baumslag, Dyer and Heller control these terms by an induction on the degree~\cite[Proposition~4.1]{BDH}. At the level of the partial resolution~\eqref{eq:resolution}, which has no intermediate homology, the first possible mixed tensor term occurs in degree $2m>m$.

\begin{lemma}\label{lem:product}
Let $\alpha,\beta$ be endomorphisms of $B$ such that $\alpha(B)$ and $\beta(B)$ commute elementwise, and let $\delta(b)=\alpha(b)\beta(b)$. For any chain lifts $\Phi_\alpha,\Phi_\beta,\Phi_\delta$ of $\alpha,\beta,\delta$, there is a finitely generated $\ZZ B$-submodule $M\subseteq Z$ such that
\[
(\Phi_\delta-\Phi_\alpha-\Phi_\beta)(Z)\subseteq M.
\]
\end{lemma}

\begin{proof}
Let $Q_*=F_*\otimes_\ZZ F_*$ with differential $d(x\otimes y)=dx\otimes y+(-1)^{|x|}x\otimes dy$, and with $\ZZ[B\times B]=\ZZ B\otimes_\ZZ\ZZ B$ acting factorwise. Each $Q_n=\bigoplus_{p+q=n}F_p\otimes_\ZZ F_q$ is a finitely generated free $\ZZ[B\times B]$-module. Write $\varepsilon_Q=\varepsilon\otimes\varepsilon\colon Q_0\to\ZZ$ for its augmentation. Let $\rho_1=\id\otimes\varepsilon$ and $\rho_2=\varepsilon\otimes\id$, where $\varepsilon$ is regarded as a chain map from $F_*$ to $\ZZ$ concentrated in degree $0$. These are chain maps $Q_*\to F_*$, semilinear with respect to the two projections $B\times B\to B$, and $\rho_1(x\otimes y)=0$ whenever $y$ has positive degree (similarly for $\rho_2$). The groups $F_p$ are free abelian and $H_0(F_*)\cong\ZZ$, so the Künneth formula has no $\mathrm{Tor}$ terms in degrees $\le m$, and gives $H_0(Q_*)\cong\ZZ$, $H_n(Q_*)=0$ for $0<n<m$, and an isomorphism
\[
(\rho_{1*},\rho_{2*})\colon H_m(Q_*)\xrightarrow{\ \cong\ }Z\oplus Z .
\]
In particular:
\begin{equation}\label{eq:kunneth}
\text{a cycle $w\in Q_m$ with $\rho_1(w)=\rho_2(w)=0$ lies in $d(Q_{m+1})$.}
\end{equation}

\emph{The product.} Since the images commute, $\eta(b,c)=\alpha(b)\beta(c)$ is a homomorphism $B\times B\to B$. Choose $h_0$ on a basis so that $\varepsilon h_0=\varepsilon_Q$. Since the $Q_n$ are free and the augmented complex $F_*\to\ZZ$ is exact at $F_0,\dots,F_{m-1}$, this extends degree by degree to $\eta$-semilinear maps $h_n\colon Q_n\to F_n$, $0\le n\le m$, commuting with the differentials. Put
\[
M_0=\ZZ B\cdot h_m d(Q_{m+1}).
\]
We have $d_mh_md=h_{m-1}dd=0$, so $M_0\subseteq Z$. Since $h_md$ is $\eta$-semilinear and $Q_{m+1}=\bigoplus_{p=1}^{m}F_p\otimes_\ZZ F_{m+1-p}$ is finitely generated, $M_0$ is finitely generated by~\eqref{eq:semilinear}. The chain maps $i_1(x)=x\otimes1$ and $i_2(x)=1\otimes x$ from $F_*$ to $Q_*$, where $1\in F_0=\ZZ B$, are semilinear with respect to $b\mapsto(b,1)$ and $b\mapsto(1,b)$. Hence $f=hi_1$ and $g=hi_2$ are chain lifts of $\alpha$ and $\beta$.

\emph{The diagonal.} Let $Q^{\mathrm{brut}}_{\le m}$ denote the brutal truncation of $Q_*$: its terms in degrees $0,\dots,m$ are unchanged and its terms in higher degrees are zero. Using $\varepsilon_Q$ in degree $0$ and exactness of the augmented complex at $Q_0,\dots,Q_{m-1}$, construct a chain map $\Delta\colon F_*\to Q^{\mathrm{brut}}_{\le m}$ with $\varepsilon_Q\Delta_0=\varepsilon$, semilinear with respect to the diagonal $B\to B\times B$. We continue to use the full complex $Q_*$, including $Q_{m+1}$, when applying~\eqref{eq:kunneth}. Then $h\Delta$ is a chain lift of $\delta$, and $\tau_a=\rho_a\Delta$ ($a=1,2$) is a $\ZZ B$-linear chain lift of $\id_B$. By Lemma~\ref{lem:comparison} there are finitely generated submodules $Y_1,Y_2\subseteq Z$ with $(\tau_a-\id)(Z)\subseteq Y_a$. For $z\in Z$, the element
\[
w=\Delta(z)-i_1\tau_1(z)-i_2\tau_2(z)
\]
is a cycle, and $\rho_1(w)=\tau_1(z)-\tau_1(z)-0=0$ because $\rho_1i_1=\id$ and $\rho_1i_2$ vanishes in positive degrees; likewise $\rho_2(w)=0$. By~\eqref{eq:kunneth}, $w\in d(Q_{m+1})$, so $h_m(w)\in M_0$. Since $hi_1\tau_1(z)=f(z)+f\bigl((\tau_1-\id)(z)\bigr)$ and similarly for $g$, we obtain
\[
(h\Delta-f-g)(Z)\subseteq M_0+\ZZ B\,f(Y_1)+\ZZ B\,g(Y_2),
\]
and the right-hand side is a finitely generated submodule of $Z$ by~\eqref{eq:semilinear}.

\emph{Arbitrary lifts.} Applying Lemma~\ref{lem:comparison} to the pairs $(\Phi_\delta,h\Delta)$, $(\Phi_\alpha,f)$ and $(\Phi_\beta,g)$, and adding the resulting finitely generated submodules to $M_0+\ZZ B\,f(Y_1)+\ZZ B\,g(Y_2)$, gives $M$.
\end{proof}

\begin{remark}
The proof gives a bound independent of $\alpha$ and $\beta$: if $F_p$ has rank $r_p$, then $M$ can be generated by $\sum_{p=1}^m r_pr_{m+1-p}+5r_m$ elements.
\end{remark}

\begin{corollary}\label{cor:row}
Let $(\alpha_v)_{v\in V}$ be endomorphisms of $B$ such that, for every $(i,j)\to k\in\Tr$, the images of $\alpha_i$ and $\alpha_j$ commute elementwise and $\alpha_k(b)=\alpha_i(b)\alpha_j(b)$ for all $b\in B$. Then every chain lift $\Phi_0$ of $\alpha_0$ maps $Z$ into a finitely generated $\ZZ B$-submodule of $Z$.
\end{corollary}

\begin{proof}
Choose chain lifts $\Phi_v$ of $\alpha_v$ for $v\ne0$. For each $\ell=(i,j)\to k$, Lemma~\ref{lem:product} gives a finitely generated submodule $M_\ell\subseteq Z$ containing $(\Phi_i+\Phi_j-\Phi_k)(Z)$; let $M=\sum_{\ell\in\Tr}M_\ell$. By~\eqref{eq:row}, applied to the abelian group of additive maps $Z\to Z$,
\[
\Phi_0(z)=\sum_{\ell=(i,j)\to k}\nu_\ell\,\bigl(\Phi_i(z)+\Phi_j(z)-\Phi_k(z)\bigr)\in M\qquad(z\in Z).\qedhere
\]
\end{proof}

\section{Ascending HNN extensions: proof of Theorem~\ref{thm:raising}}\label{sec:HNN}

\begin{proof}[Proof of Theorem~\ref{thm:raising}]
By Britton's lemma, $B$ embeds in $E$. Since $\ZZ E$ is a free right $\ZZ B$-module, the functor $\ZZ E\otimes_{\ZZ B}-$ is exact. Put $P_i=\ZZ E\otimes_{\ZZ B}F_i$ and $Z_E=\ZZ E\otimes_{\ZZ B}Z$. Then $P_m\to\cdots\to P_0\to\ZZ[E/B]\to0$ is exact, the $P_i$ are finitely generated free $\ZZ E$-modules, and $Z_E$ is the kernel of $d_m\colon P_m\to P_{m-1}$.

Choose a chain lift $\Phi$ of $\varphi$ satisfying $\Phi(Z)\subseteq M$ for a finitely generated $\ZZ B$-submodule $M\subseteq Z$, and define $T_i\colon P_i\to P_i$ by
\[
T_i(a\otimes x)=at\otimes\Phi_i(x).
\]
This is well defined: for $b\in B$ we have $bt=t\varphi(b)$ in $E$, so $abt\otimes\Phi_i(x)=at\varphi(b)\otimes\Phi_i(x)=at\otimes\Phi_i(bx)$. The maps $T_i$ are $\ZZ E$-linear, commute with the differentials, and induce $T(gB)=gtB$ on $\ZZ[E/B]=\ZZ E\otimes_{\ZZ B}\ZZ$ because $\varepsilon\Phi_0=\varepsilon$. Let $\mu_1,\dots,\mu_r$ generate $M$. If $z\in Z$ and $\Phi_m(z)=\sum_p\lambda_p\mu_p$ with $\lambda_p\in\ZZ B$, then $T_m(a\otimes z)=\sum_pat\lambda_p\otimes \mu_p$. Hence
\begin{equation}\label{eq:TinM}
T_m(Z_E)\subseteq\ZZ E\,(1\otimes \mu_1)+\dots+\ZZ E\,(1\otimes \mu_r).
\end{equation}

The sequence
\begin{equation}\label{eq:coset}
0\to\ZZ[E/B]\xrightarrow{\,1-T\,}\ZZ[E/B]\xrightarrow{\,\varepsilon\,}\ZZ\to0
\end{equation}
is exact. Indeed, the height is constant on each coset $gB$ and $T$ raises it by one, so the terms of minimal height of a nonzero element of $\ZZ[E/B]$ survive in its image under $1-T$. The cokernel of $1-T$ is the quotient of $\ZZ[E/B]$ by the relations $gB=gtB$; together with $gbB=gB$ for $b\in B$, these identify $gB$ with $gxB$ for every $x\in B\cup\{t^{\pm1}\}$, hence identify all cosets, since $B$ and $t$ generate $E$.

The idea is that, modulo $(1-T_m)Z_E$, every $\zeta\in Z_E$ is congruent to $T_m\zeta$, which lies in the finitely generated submodule~\eqref{eq:TinM}. We make this precise with a mapping cone.

Let $C_*$ be the brutal truncation in degrees $\le m$ of the mapping cone of $1-T$: $C_0=P_0$, $C_j=P_j\oplus P_{j-1}$ for $1\le j\le m$, with differentials
\begin{align*}
D_1(x,y)&=d_1x+(1-T_0)y,\\
D_j(x,y)&=\bigl(d_jx+(1-T_{j-1})y,\,-d_{j-1}y\bigr)\qquad(2\le j\le m),
\end{align*}
augmented by the composite $C_0=P_0\to\ZZ[E/B]\xrightarrow{\varepsilon}\ZZ$. Since $T$ commutes with the differentials, $D_{j-1}D_j=0$. Let $\xi_1,\dots,\xi_s$ be a basis of $F_m$. We show that $C_*$ is exact in degrees $<m$ and that $\ker D_m$ is generated by the elements
\begin{equation}\label{eq:generators}
(1\otimes \mu_p,\,0)\quad(1\le p\le r),\qquad\bigl((1-T_m)(1\otimes\xi_q),\,-d_m(1\otimes\xi_q)\bigr)\quad(1\le q\le s).
\end{equation}
Mapping a free $\ZZ E$-module of rank $r+s$ onto $\ker D_m$ then extends $C_*$ to a partial resolution of $\ZZ$ by finitely generated free $\ZZ E$-modules in degrees $0,\dots,m+1$, so $E$ is of type $\FP{m+1}$.

\emph{Degree $0$.} Let $x\in P_0$ have augmentation $0$. By~\eqref{eq:coset}, its image in $\ZZ[E/B]$ is $(1-T)\bar y$ for the image $\bar y$ of some $y\in P_0$. Then $x-(1-T_0)y$ lies in the kernel of $P_0\to\ZZ[E/B]$, which is $d_1(P_1)$, so $x\in D_1(C_1)$.

\emph{Degrees $1\le j\le m$.} Let $(x,y)\in C_j$ with $D_j(x,y)=0$. We first show that $y=d_jz$ for some $z\in P_j$. If $j=1$, mapping $D_1(x,y)=0$ to $\ZZ[E/B]$ gives $(1-T)\bar y=0$, so $\bar y=0$ by~\eqref{eq:coset} and $y\in d_1(P_1)$. If $j\ge2$, then $d_{j-1}y=0$ and $y$ is a boundary because $H_{j-1}(P_*)=0$. Now $d_j\bigl(x+(1-T_j)z\bigr)=0$. If $j<m$, then $x+(1-T_j)z=d_{j+1}w$ for some $w$, and $(x,y)=D_{j+1}(w,-z)$. If $j=m$, then $\zeta=x+(1-T_m)z$ lies in $Z_E$ and
\[
(x,y)=(\zeta,0)+\bigl((1-T_m)(-z),\,-d_m(-z)\bigr).
\]
The elements $\bigl((1-T_m)x',-d_mx'\bigr)$ with $x'\in P_m$ form the image of a $\ZZ E$-linear map $P_m\to C_m$, which is generated by the second family in~\eqref{eq:generators}. For $\zeta\in Z_E$ we have $d_m\zeta=0$, hence
\[
(\zeta,0)=(T_m\zeta,0)+\bigl((1-T_m)\zeta,\,-d_m\zeta\bigr),
\]
where the first term lies in the span of the first family in~\eqref{eq:generators} by~\eqref{eq:TinM}, and the second in the span of the second family.
\end{proof}

Combining Theorem~\ref{thm:raising} with Corollary~\ref{cor:row}, we obtain:

\begin{corollary}\label{cor:raising}
Let $m\ge1$ and let $B$ be a group of type $\FP{m}$. Let $(\alpha_v)_{v\in V}$ be endomorphisms of $B$ satisfying the hypothesis of Corollary~\ref{cor:row}, with $\alpha_0$ injective. Then the ascending HNN extension of $B$ along $\alpha_0$ contains $B$ and is of type $\FP{m+1}$.\qed
\end{corollary}

\section{Proofs of Theorem~\ref{thm:A}, Theorem~\ref{thm:B} and Corollary~\ref{cor:C}}

We fix an oracle $\orc\subseteq\NN$. The first lemma is the Higman--Neumann--Neumann embedding theorem~\cite{HNN} (see~\cite[IV.3]{LS}), in a form that keeps track of presentations.

\begin{lemma}\label{lem:HNN}
Every countable group with an $\orc$-r.e.\ presentation embeds in a two-generator $\orc$-recursively presented group $G^*$.
\end{lemma}

\begin{proof}
Write $G=\langle x_1,x_2,\dots\colon\rels\rangle$, repeating generators if necessary, and let $L=\langle a,b\rangle$ be free of rank two. In $G*L$ put
\[
u_0=b,\quad u_i=a^{-i}ba^i;\qquad v_0=a,\quad v_i=x_ib^{-i}ab^i\qquad(i\ge1).
\]
The $u_i$ freely generate a free subgroup of $L$. So do the $v_i$, since the retraction $G*L\to L$ killing $G$ maps them to the free basis $\{b^{-i}ab^i\colon i\ge0\}$ of a subgroup of $L$. Hence $G*L$, and so $G$, embeds in
\[
G^*=\langle G*L,\,t\colon t^{-1}u_it=v_i,\ i\ge0\rangle .
\]
The defining relations give $a=t^{-1}bt$ and $x_i=t^{-1}a^{-i}ba^it\,b^{-i}a^{-1}b^i$ for $i\ge1$. Eliminating $a$ and the $x_i$ presents $G^*$ on $b,t$ with relators obtained from $\rels$ by substituting these words, which is an $\orc$-r.e.\ set.
\end{proof}

\begin{lemma}\label{lem:initial}
There is a two-generator $\orc$-recursively presented group $B_{1,\orc}$ that contains a copy of every finitely generated $\orc$-recursively presented group. In particular, $B_{1,\orc}$ is of type $\FP{1}$.
\end{lemma}

\begin{proof}
For $k\ge1$ and an oracle Turing machine $\mathcal{M}$, let $G_{k,\mathcal{M}}$ be the group generated by $x_1,\dots,x_k$ subject to those outputs of $\mathcal{M}$ with oracle $\orc$ that are words in $x_1^{\pm1},\dots,x_k^{\pm1}$. Every finitely generated $\orc$-recursively presented group is isomorphic to some $G_{k,\mathcal{M}}$. Using disjoint alphabets and dovetailing the enumerations, the free product $\Omega_\orc$ of all these groups has an $\orc$-r.e.\ presentation and contains each of them. By Lemma~\ref{lem:HNN}, $\Omega_\orc$ embeds in a two-generator $\orc$-recursively presented group $B_{1,\orc}$. Finite generation is equivalent to type $\FP{1}$.
\end{proof}

\begin{theorem}\label{thm:oracle}
For every oracle $\orc\subseteq\NN$ and every integer $n\ge1$ there is an $\orc$-recursively presented group $B_{n,\orc}$ of type $\FP{n}$ that contains a copy of every finitely generated $\orc$-recursively presented group. These groups can be chosen so that $B_{1,\orc}\le B_{2,\orc}\le\cdots$.
\end{theorem}

\begin{proof}
The group $B_{1,\orc}$ is given by Lemma~\ref{lem:initial}. Suppose that $B=B_{m,\orc}$ has been constructed for some $m\ge1$. By Lemma~\ref{lem:presentations}, $K_\Tr(B)$ is finitely generated and $\orc$-recursively presented, so there is an embedding $\kappa\colon K_\Tr(B)\to B$. Put $\alpha_v=\kappa\circ\iota_v$ for $v\in V$. The relations~\eqref{eq:KSigma} show that the family $(\alpha_v)$ satisfies the hypothesis of Corollary~\ref{cor:row}, and $\alpha_0$ is injective by Lemma~\ref{lem:injective}. By Corollary~\ref{cor:raising},
\[
B_{m+1,\orc}=\langle B,\,t_m\colon t_m^{-1}bt_m=\alpha_0(b),\ b\in B\rangle
\]
is of type $\FP{m+1}$ and contains $B$, hence every finitely generated $\orc$-recursively presented group. It is $\orc$-recursively presented by Subsection~\ref{sec:oracle}.
\end{proof}

\begin{proof}[Proof of Theorem~\ref{thm:A}]
Let $G$ be a countable group and choose $\orc$ such that $G$ has an $\orc$-r.e.\ presentation (see Subsection~\ref{sec:oracle}). By Lemma~\ref{lem:HNN}, $G$ embeds in a finitely generated $\orc$-recursively presented group, hence in $B_{1,\orc}\le B_{n,\orc}$, which is of type $\FP{n}$ by Theorem~\ref{thm:oracle}.
\end{proof}

\begin{remark}\label{rem:infinity}
For $n=0$ every group is of type $\FP{0}$, and for $n=1$ the required embedding follows from Lemma~\ref{lem:HNN}. As constructed here, the group obtained depends on $G$ and on $n$, and passing to a union does not give a group of type $\FP{\infty}$: the union of the chain $B_{1,\orc}\le B_{2,\orc}\le\cdots$ is not finitely generated, since the height of $B_{m+1,\orc}$ vanishes on $B_{m,\orc}$ and takes the value $1$ at $t_m$, so each inclusion is proper, and a finite subset of the union lies in some $B_{m,\orc}$. Nevertheless, $B_{2,\orc}$ is already of type $\FP{\infty}$; see Remark~\ref{app:B2}.
\end{remark}

\subsection{Homotopical finiteness}\label{sec:B}
We now turn to Theorem~\ref{thm:B} and its corollary. We use the following classical result attributed to Wall~\cite{Wall}; see~\cite[VIII.7]{Bro}. A short proof is given in~\cite[Lemma~2.1]{FFZ}.
\begin{lemma}\label{lem:Fn}
A finitely presented group of type $\FP{n}$ is of type $\Ft{n}$.\qed
\end{lemma}

\begin{proof}[Proof of Theorem~\ref{thm:B}]
Let $\Omega_\emptyset$ be the recursively presented free product constructed in Lemma~\ref{lem:initial}. By Lemma~\ref{lem:HNN} and Higman's embedding theorem~\cite{Hig}, it embeds in a finitely presented group $U_2$. Every countable recursively presented group embeds, by Lemma~\ref{lem:HNN}, in a finitely generated recursively presented group, and the latter is isomorphic to a free factor of $\Omega_\emptyset$. Thus $U_2$ contains a copy of every countable recursively presented group.

Suppose that $U_m$ is finitely presented, of type $\Ft{m}$, and contains $U_2$. In particular, it contains a copy of every finitely presented group. By Lemma~\ref{lem:presentations}, $K_\Tr(U_m)$ is finitely presented and hence embeds in $U_m$. As in the proof of Theorem~\ref{thm:oracle}, composing this embedding with the maps $\iota_v$ gives endomorphisms satisfying the hypotheses of Corollary~\ref{cor:raising}, and the resulting ascending HNN extension $U_{m+1}$ is of type $\FP{m+1}$. It is finitely presented, being obtained from a finite presentation of $U_m$ by adding one generator and finitely many relators, so it is of type $\Ft{m+1}$ by Lemma~\ref{lem:Fn}. It contains $U_m$, hence $U_2$, hence a copy of every countable recursively presented group.
\end{proof}

\begin{proof}[Proof of Corollary~\ref{cor:C}]
A group of type $\Ft{n}$ with $n\ge2$ is finitely presented, so every subgroup of such a group embeds in a finitely presented group. Conversely, if a countable group $G$ embeds in a finitely presented group $P$, then Theorem~\ref{thm:B} embeds $P$, and hence $G$, in a group of type $\Ft{n}$. For the final assertion, a finitely generated subgroup of a finitely presented group is recursively presented: choose words for its generators and enumerate the relations among them that follow from the finite presentation. The converse is Higman's embedding theorem~\cite{Hig}.
\end{proof}
  
\subsection{Two-generator versions}
The finite-generator version of Lemma~\ref{lem:HNN} preserves the finiteness properties considered here.

\begin{lemma}\label{lem:two-generators}
Let $n\ge2$. Every finitely generated group of type $\FP{n}$ embeds in a two-generator group of type $\FP{n}$. If the original group is of type $\Ft{n}$, the two-generator group can also be chosen of type $\Ft{n}$.
\end{lemma}

\begin{proof}
Let $G=\langle x_1,\dots,x_k\rangle$, with $k\ge1$, and use the construction of Lemma~\ref{lem:HNN} with only the indices $0,\dots,k$. The subgroups generated by $u_0,\dots,u_k$ and by $v_0,\dots,v_k$ are free of rank $k+1$. Thus
\[
J=\langle G*\langle a,b\rangle,\,t\colon t^{-1}u_it=v_i,\ 0\le i\le k\rangle
\]
contains $G$ by Britton's lemma and is generated by $b,t$, by the same elimination formulas as before.

We recall why the finite-rank edge groups preserve type $\FP{n}$. For a group $\Gamma$ that is the fundamental group of a finite graph of groups, the augmented cellular chain sequence of its Bass--Serre tree is
\[
0\to\bigoplus_e\ZZ\Gamma\otimes_{\ZZ G_e}\ZZ
\to\bigoplus_v\ZZ\Gamma\otimes_{\ZZ G_v}\ZZ
\to\ZZ\to0,
\]
where one edge is chosen from each unoriented edge orbit. Induce free resolutions of $\ZZ$ over the vertex and edge groups to $\ZZ\Gamma$, lift the first map in this sequence to a chain map, and take its mapping cone. Exactness of induction and the homology sequence of the cone show that this is a free resolution of $\ZZ$. Its degree-$j$ term is the sum of the induced degree-$j$ vertex terms and degree-$(j-1)$ edge terms. Hence $\Gamma$ is of type $\FP{n}$ if its vertex groups are of type $\FP{n}$ and its edge groups are of type $\FP{n-1}$.

Apply this first to the free product $G*\langle a,b\rangle$, whose edge group is trivial, and then to the HNN extension defining $J$, whose edge group is free of finite rank. This proves that $J$ is of type $\FP{n}$. If $G$ is of type $\Ft{n}$, it is finitely presented, and so is $J$: the construction adds finitely many generators and relators. Lemma~\ref{lem:Fn} then shows that $J$ is of type $\Ft{n}$.
\end{proof}

\begin{corollary}\label{cor:two-generators}
Let $n\ge2$. Every countable group embeds in a two-generator group of type $\FP{n}$. Moreover, there is a two-generator group of type $\Ft{n}$ containing a copy of every countable recursively presented group.
\end{corollary}

\begin{proof}
Apply Lemma~\ref{lem:two-generators} to the groups supplied by Theorems~\ref{thm:A} and~\ref{thm:B}, which are finitely generated because $n\ge2$.
\end{proof}

\appendix
\section{Embedding in a group of type $\FP{\infty}$}\label{app:infinity}

The preprint~\cite{OAI}, released on 6 October 2026, constructs a group of type $\Ft{\infty}$ containing a copy of every finitely presented group. Its controlled-filling argument, combined with the triple system of Proposition~\ref{prop:certificate} and the oracle-universal groups of Lemma~\ref{lem:initial}, yields the following result, which contains Theorem~\ref{thm:O}.

\begin{theorem}\label{app:main}
Let $\orc\subseteq\NN$. There is a two-generator $\orc$-recursively presented group of type $\FP{\infty}$ that contains a copy of every countable $\orc$-recursively presented group. In particular, every countable group embeds in a two-generator group of type $\FP{\infty}$.
\end{theorem}

The group depends on the oracle, but not on any finiteness degree; see Remark~\ref{app:B2}.

\subsection{Relation to~\cite{OAI}}
We adapt three ingredients of~\cite{OAI}: controlled additivity~\cite[Lemma~6.1]{OAI}, the inductive construction of uniform fillings~\cite[Proposition~6.2]{OAI}, and the passage to homology with coefficients in products of free modules~\cite[Sections~7.2--7.4]{OAI}. The two idempotent diagrams of~\cite{OAI} are replaced by the integral identity~\eqref{eq:row}, which annihilates the relevant classes directly; finite generation of the base group, rather than finite presentability, starts the induction; the fillings are organised as controlled partial contractions of powers of the endomorphism, which serve both in the induction and in the coefficient argument; and the geometric argument of~\cite[Lemma~7.1]{OAI} is replaced by an algebraic one. We do not use~\cite[Theorem~5.1]{OAI}, whose hypothesis that the endomorphism factors through a finitely presented group is not available here.

The core of the argument is the following.

\begin{proposition}\label{app:torus}
Let $B$ be a finitely generated group, and let $(\alpha_v)_{v\in V}$ be endomorphisms of $B$ satisfying the hypothesis of Corollary~\ref{cor:row}, with $\varphi=\alpha_0$ injective. Then
\[
E=\langle B,t\colon t^{-1}bt=\varphi(b),\ b\in B\rangle
\]
is of type $\FP{\infty}$.
\end{proposition}

We fix $B$, $(\alpha_v)$, $\varphi$ and $E$ as in Proposition~\ref{app:torus} until its proof is complete.

\subsection{Controlled chains}
For a group $\Gamma$, let $C_*(\Gamma)$ be the normalized homogeneous bar complex: $C_q(\Gamma)$ is the free abelian group on the tuples $(\gamma_0,\dots,\gamma_q)\in\Gamma^{q+1}$ with $\gamma_{i-1}\ne\gamma_i$ for all $i$, a tuple with two consecutive equal entries is read as $0$, and $\partial$ is the alternating sum of the deletions. With diagonal left translation and the augmentation $\varepsilon(\gamma)=1$, it is a free $\ZZ\Gamma$-resolution of $\ZZ$~\cite[I.5]{Bro}. A homomorphism $\theta\colon\Gamma\to\Gamma'$ induces, entrywise, a $\theta$-semilinear chain map $\theta_*$. The constant map $B\to\{1\}$ likewise induces a chain map $c_*$ of $C_*(B)$; thus $c_*(b)=(1)$ in degree $0$, and $c_*=0$ in positive degrees.

A \emph{bound} in $\Gamma$ is a finite symmetric subset $A\subseteq\Gamma$ containing $1$. Let $C^A_*(\Gamma)\subseteq C_*(\Gamma)$ be spanned by the tuples with $\gamma_i^{-1}\gamma_j\in A$ for all $i,j$. It is a subcomplex, invariant under left translation, and $C^A_q(\Gamma)$ is a free $\ZZ\Gamma$-module on the finite set $\mathcal{S}_q(A)$ of its nondegenerate tuples with $\gamma_0=1$; this set is part of a $\ZZ\Gamma$-basis of $C_q(\Gamma)$, so $C^A_q(\Gamma)$ is a direct summand of $C_q(\Gamma)$. Every chain lies in $C^A_*(\Gamma)$ for some bound $A$, a finite union of bounds is a bound, $\theta_*\bigl(C^A_*(\Gamma)\bigr)\subseteq C^{\theta(A)}_*(\Gamma')$ for every homomorphism $\theta$, and $A_1\times A_2$ is a bound in $\Gamma_1\times\Gamma_2$. An additive map defined on $C^A_q(\Gamma)$ is \emph{controlled} if its image lies in $C^{A'}_*(\Gamma')$ for some bound $A'$. This restricts the shape of the simplices that occur, but neither their number, nor their coefficients, nor the position of their vertices.

\begin{definition}\label{app:contraction-def}
Let $\theta$ be an endomorphism of $B$, let $A$ be a bound in $B$ and let $n\ge0$. A \emph{controlled $n$-contraction} of $\theta$ over $A$ is a family of controlled additive maps $h_q\colon C^A_q(B)\to C_{q+1}(B)$, $0\le q\le n$, such that, with $h_{-1}=0$,
\begin{equation}\label{app:contraction-eq}
\partial h_q+h_{q-1}\partial=\theta_*-c_*\quad\text{on }C^A_q(B)\qquad(0\le q\le n).
\end{equation}
\end{definition}

If $(h_q)$ is a controlled $n$-contraction of $\theta$ over $A$, then $(\psi_*h_q)$ is one of $\psi\theta$ for every endomorphism $\psi$ of $B$, since $\psi_*c_*=c_*$.

\subsection{Controlled additivity}
Let $\mathrm{AW}\colon C_*(\Gamma)\to C_*(\Gamma)\otimes_\ZZ C_*(\Gamma)$ be the Alexander--Whitney map,
\[
\mathrm{AW}(\gamma_0,\dots,\gamma_q)=\sum_{p=0}^q(\gamma_0,\dots,\gamma_p)\otimes(\gamma_p,\dots,\gamma_q),
\]
and let $\nabla\colon C_*(\Gamma_1)\otimes_\ZZ C_*(\Gamma_2)\to C_*(\Gamma_1\times\Gamma_2)$ be the shuffle map: on tuples of lengths $p+1$ and $q+1$, it is the sum, over the monotone lattice paths from $(0,0)$ to $(p,q)$, of the tuples of pairs of entries visited by the path, with sign $(-1)^N$, where $N$ is the number of pairs consisting of a vertical step and a later horizontal step. Both are natural chain maps of normalized chains, with the differential $\partial(x\otimes y)=\partial x\otimes y+(-1)^{|x|}x\otimes\partial y$ on tensor products. The map $\mathrm{AW}$ takes $C^A_*(\Gamma)$ to $C^A_*(\Gamma)\otimes C^A_*(\Gamma)$, since it only takes faces, and $\nabla$ takes $C^{A_1}_*(\Gamma_1)\otimes C^{A_2}_*(\Gamma_2)$ to $C^{A_1\times A_2}_*(\Gamma_1\times\Gamma_2)$. Write $\Delta,i_1,i_2\colon B\to B\times B$ for $b\mapsto(b,b)$, $(b,1)$, $(1,b)$. For every chain $w$ of $C_*(B)$,
\begin{equation}\label{app:shuffle-point}
\nabla\bigl(w\otimes(1)\bigr)=(i_1)_*w,\qquad\nabla\bigl((1)\otimes w\bigr)=(i_2)_*w.
\end{equation}

\begin{lemma}\label{app:EZ}
There are additive maps $\Psi\colon C_q(B)\to C_{q+1}(B\times B)$, $q\ge0$, such that
\[
\partial\Psi+\Psi\partial=\Delta_*-\nabla\,\mathrm{AW}
\]
and $\Psi\bigl(C^A_*(B)\bigr)\subseteq C^{A\times A}_{*+1}(B\times B)$ for every bound $A$.
\end{lemma}

\begin{proof}
For $x\in B\times B$ and a chain $u$, let $x*u$ be obtained by prepending $x$ to every tuple; then $\partial(x*u)=u$ whenever $u$ is a cycle of positive degree. In degree $0$, both $\Delta_*$ and $\nabla\,\mathrm{AW}$ map $(b)$ to $\bigl((b,b)\bigr)$; put $\Psi=0$ there. Let $q\ge1$, and suppose that $\Psi$ has been defined in degrees $<q$, satisfies the identity there, and maps every tuple to a chain whose entries are pairs of entries of that tuple. Let $s=(b_0,\dots,b_q)$ be a nondegenerate tuple with set of entries $W$. The chain $u_s=\Delta_*s-\nabla\,\mathrm{AW}(s)-\Psi\partial s$ has entries in $W\times W$, and is a cycle by the identity in degree $q-1$. Put $\Psi(s)=(b_0,b_0)*u_s$. If $s\in C^A_q(B)$, then $b^{-1}b'\in A$ for all $b,b'\in W$, whence the last assertion.
\end{proof}

\begin{lemma}[Controlled additivity; cf.~{\cite[Lemma~6.1]{OAI}}]\label{app:additivity}
Let $n\ge1$, let $A$ be a bound in $B$, and let $\theta$ be an endomorphism of $B$ admitting a controlled $(n-1)$-contraction over $A$. Let $\alpha,\beta$ be endomorphisms of $B$ whose images commute elementwise, and put $\delta(b)=\alpha(b)\beta(b)$. Then there is a bound $A'$ such that, for every cycle $z\in C^A_n(B)$,
\[
(\delta\theta)_*z-(\alpha\theta)_*z-(\beta\theta)_*z\in\partial C^{A'}_{n+1}(B).
\]
\end{lemma}

\begin{proof}
Let $(h_q)_{q<n}$ be the contraction, put $h_q=0$ for $q\ge n$, and let $\Theta=\theta_*-\partial h-h\partial$ on $C^A_*(B)$. This is a chain map, with $\Theta_0=c_*$, $\Theta_q=0$ for $0<q<n$, and $\Theta_n=\theta_*-h_{n-1}\partial$; in particular $\Theta_nz=\theta_*z$ for every cycle $z$. On $C^A_*(B)\otimes C^A_*(B)$ define $\Xi$, of degree $1$, by
\[
\Xi(x\otimes y)=h(x)\otimes\theta_*(y)+(-1)^{|x|}\,\Theta(x)\otimes h(y).
\]
Since $\partial h+h\partial=\theta_*-\Theta$, direct expansion gives $\partial\Xi+\Xi\partial=\theta_*\otimes\theta_*-\Theta\otimes\Theta$. Put $\Lambda=\Psi\theta_*+\nabla\,\Xi\,\mathrm{AW}$ on $C^A_*(B)$. Lemma~\ref{app:EZ} and the naturality $\mathrm{AW}\,\theta_*=(\theta_*\otimes\theta_*)\mathrm{AW}$ give
\[
\partial\Lambda+\Lambda\partial=\Delta_*\theta_*-\nabla(\Theta\otimes\Theta)\mathrm{AW}.
\]
Let $A_1$ be a bound containing $\theta(A)$ and the images of $h_q$ ($q<n$) and $\Theta_q$ ($q\le n$) on $C^A_*(B)$. Only these maps enter $\Lambda$ on $C^A_n(B)$, so $\Lambda$ maps $C^A_n(B)$ into $C^{A_1\times A_1}_{n+1}(B\times B)$.

Let $z\in C^A_n(B)$ be a cycle. In $(\Theta\otimes\Theta)\mathrm{AW}(z)$, the terms of bidegree $(p,n-p)$ with $0<p<n$ vanish, and the two remaining terms sum to $\theta_*z\otimes(1)+(1)\otimes\theta_*z$. By~\eqref{app:shuffle-point},
\[
\partial\Lambda z=\Delta_*\theta_*z-(i_1)_*\theta_*z-(i_2)_*\theta_*z.
\]
Since the images of $\alpha$ and $\beta$ commute, $\mu(x,y)=\alpha(x)\beta(y)$ defines a homomorphism $\mu\colon B\times B\to B$ with $\mu\Delta=\delta$, $\mu i_1=\alpha$ and $\mu i_2=\beta$. Hence
\[
\partial\mu_*\Lambda z=(\delta\theta)_*z-(\alpha\theta)_*z-(\beta\theta)_*z.
\]
Finally, $A'=\alpha(A_1)\beta(A_1)$ is a bound, symmetric because the images commute, and $\mu_*$ maps $C^{A_1\times A_1}_*(B\times B)$ into $C^{A'}_*(B)$, since $\mu(x,y)^{-1}\mu(x',y')=\alpha(x^{-1}x')\,\beta(y^{-1}y')$.
\end{proof}

\subsection{Contractions of powers of $\varphi$}
\begin{proposition}[cf.~{\cite[Proposition~6.2]{OAI}}]\label{app:contraction}
For every $n\ge0$ and every bound $A$ in $B$, some power $\varphi^r$ with $r\ge1$ admits a controlled $n$-contraction over $A$.
\end{proposition}

\begin{proof}
For $n=0$, fix a finite symmetric generating set $S$ of $B$ containing $1$. For each $u\in B$, choose $s_1,\dots,s_\ell\in S$ with $\varphi(u)=s_1\cdots s_\ell$, and put
\[
h_0(u)=\sum_{k=1}^{\ell}(s_1\cdots s_{k-1},\,s_1\cdots s_k)\in C^S_1(B).
\]
Then $\partial h_0(u)=(\varphi(u))-(1)$, so $(h_0)$ is a controlled $0$-contraction of $\varphi$ over every bound. The lengths $\ell$ need not be bounded.

Let $n\ge1$, and assume the assertion in degree $n-1$ for all bounds. Let $(h_q)_{q<n}$ be a controlled $(n-1)$-contraction of $\varphi^r$ over $A$. For each nondegenerate tuple $s$ in $C^A_n(B)$, the chain $w_s=\varphi^r_*s-h_{n-1}\partial s$ is a cycle, since
\[
\partial w_s=\varphi^r_*\partial s-(\varphi^r_*-c_*-h_{n-2}\partial)\partial s=c_*\partial s=0;
\]
indeed $c_*$ vanishes in positive degrees and $\varepsilon(\partial s)=0$. The cycles $w_s$ all lie in $C^{A_1}_n(B)$ for one bound $A_1$, because $\varphi^r_*$ and $h_{n-1}$ are controlled. By hypothesis, some $\varphi^{r'}$ admits a controlled $(n-1)$-contraction over $A_1$. Apply Lemma~\ref{app:additivity} with $\theta=\varphi^{r'}$ to $(\alpha,\beta)=(\alpha_i,\alpha_j)$ for each triple $(i,j)\to k$ of $\Tr$, and let $A'$ be the union of the resulting bounds. For a cycle $w\in C^{A_1}_n(B)$ put $y_v=(\alpha_v\varphi^{r'})_*w$. Then $y_i+y_j-y_k\in\partial C^{A'}_{n+1}(B)$ for every $(i,j)\to k$, and~\eqref{eq:row}, applied in the abelian group $C_n(B)$, gives
\[
\varphi^{r'+1}_*w=y_0=\sum_{\ell=(i,j)\to k}\nu_\ell\,(y_i+y_j-y_k)\in\partial C^{A'}_{n+1}(B).
\]
In particular, $\varphi^{r'+1}_*w_s=\partial y_s$ for some $y_s\in C^{A'}_{n+1}(B)$. Extend $s\mapsto y_s$ additively to $h_n\colon C^A_n(B)\to C^{A'}_{n+1}(B)$. Then $(\varphi^{r'+1}_*h_0,\dots,\varphi^{r'+1}_*h_{n-1},h_n)$ is a controlled $n$-contraction of $\varphi^{r+r'+1}$ over $A$: in degrees $<n$ this was observed after Definition~\ref{app:contraction-def}, and in degree $n$
\[
\partial h_n(s)+\varphi^{r'+1}_*h_{n-1}\partial s=\varphi^{r'+1}_*(w_s+h_{n-1}\partial s)=\varphi^{r+r'+1}_*s.\qedhere
\]
\end{proof}

The power and the bounds depend on $n$ and $A$, but not on individual chains; this uniformity is what Lemma~\ref{app:product-vanishing} requires. Neither is $\varphi$ assumed to commute with the $\alpha_v$, since all relations are precomposed with the same power of $\varphi$, nor need the contractions be equivariant.

\subsection{Coefficients in products of free modules}
In this subsection coefficient modules are right modules. For a right $\ZZ E$-module $N$, the groups $H_*(B,N)$ are the homology of $N\otimes_{\ZZ B}C_*(B)$, on which
\[
T(x\otimes\sigma)=xt\otimes\varphi_*\sigma
\]
is a well-defined chain endomorphism, since $bt=t\varphi(b)$ in $E$.

\begin{lemma}\label{app:hnn-sequence}
For every right $\ZZ E$-module $N$ there is a long exact sequence
\[
\cdots\to H_q(B,N)\xrightarrow{1-T}H_q(B,N)\to H_q(E,N)\to H_{q-1}(B,N)\xrightarrow{1-T}\cdots.
\]
\end{lemma}

\begin{proof}
The complex $P_*=\ZZ E\otimes_{\ZZ B}C_*(B)$ is a free $\ZZ E$-resolution of $\ZZ[E/B]$, and $\tilde T(g\otimes\sigma)=gt\otimes\varphi_*\sigma$ is a $\ZZ E$-linear chain endomorphism of $P_*$ lifting $T(gB)=gtB$. By the exactness of~\eqref{eq:coset}, the mapping cone of $1-\tilde T$ is a free $\ZZ E$-resolution of $\ZZ$. Tensoring with $N$ identifies $N\otimes_{\ZZ E}P_*$ with $N\otimes_{\ZZ B}C_*(B)$ and $1\otimes\tilde T$ with $T$; the long exact sequence of the cone is the assertion.
\end{proof}

\begin{lemma}[cf.~{\cite[Sections~7.3--7.4]{OAI}}]\label{app:product-vanishing}
Let $I$ be a set and $N=(\ZZ E)^I$. Then $T$ is locally nilpotent on $H_q(B,N)$ for every $q\ge1$. Consequently $H_q(E,N)=0$ for every $q\ge2$.
\end{lemma}

\begin{proof}
For every bound $A$ and every $p$, the module $C^A_p(B)$ is free on the finite set $\mathcal{S}_p(A)$, and finite direct sums commute with products, so the coordinate map
\begin{equation}\label{app:finite-types}
\pi\colon N\otimes_{\ZZ B}C^A_p(B)\xrightarrow{\ \cong\ }\prod_{\lambda\in I}\ZZ E\otimes_{\ZZ B}C^A_p(B)
\end{equation}
is an isomorphism, compatible with $\partial$, with $T$ and with enlarging $A$. Moreover $N\otimes_{\ZZ B}C^A_*(B)$ is a subcomplex of $N\otimes_{\ZZ B}C_*(B)$, because $C^A_p(B)$ is a direct summand of $C_p(B)$.

Let $q\ge1$, let $z$ be a $q$-cycle, and choose a bound $A$ with $z\in N\otimes_{\ZZ B}C^A_q(B)$. By Proposition~\ref{app:contraction}, some $\varphi^r$ admits a controlled $q$-contraction $(h_p)_{p\le q}$ over $A$; let $A'\supseteq A\cup\varphi^r(A)$ be a bound containing the images of the $h_p$. Fix representatives of the cosets $E/B$. As a right $\ZZ B$-module, $\ZZ E$ is the direct sum of the submodules $g\ZZ B$, so every element of $\ZZ E\otimes_{\ZZ B}C^A_p(B)$ can be written uniquely as $\sum_gg\otimes x_g$ with $x_g\in C^A_p(B)$. Put
\[
H\Bigl(\sum_gg\otimes x_g\Bigr)=\sum_ggt^r\otimes h_p(x_g)\in\ZZ E\otimes_{\ZZ B}C^{A'}_{p+1}(B)\qquad(0\le p\le q).
\]
Since $\partial$ acts componentwise and $c_*$ vanishes in positive degrees,~\eqref{app:contraction-eq} gives $\partial H+H\partial=T^r$ on $\ZZ E\otimes_{\ZZ B}C^A_p(B)$ for $1\le p\le q$. Applying $H$ in every coordinate and using~\eqref{app:finite-types} for $A$ and $A'$, we obtain maps $H^I\colon N\otimes_{\ZZ B}C^A_p(B)\to N\otimes_{\ZZ B}C^{A'}_{p+1}(B)$ satisfying $\partial H^I+H^I\partial=T^r$ in the same degrees, since both sides lie in $N\otimes_{\ZZ B}C^{A'}_p(B)$ and agree in every coordinate. In particular $T^rz=\partial H^Iz$, so $T^r[z]=0$.

For a locally nilpotent $T$, the map $1-T$ is bijective, with inverse $x\mapsto\sum_{k\ge0}T^kx$. For $q\ge2$, the maps $1-T$ in degrees $q$ and $q-1$ of the sequence of Lemma~\ref{app:hnn-sequence} are therefore bijective, and $H_q(E,N)=0$.
\end{proof}

\subsection{From vanishing to finiteness}
The following is a special case of the Bieri--Eckmann criterion~\cite[Section~1]{Bie}; we include the short argument.

\begin{lemma}\label{app:syzygies}
Let $m\ge1$, and let $\Gamma$ be a group of type $\FP{m}$ such that $H_q(\Gamma,(\ZZ\Gamma)^I)=0$ for every set $I$ and every $q\ge m$. Then $\Gamma$ is of type $\FP{\infty}$.
\end{lemma}

\begin{proof}
Write $R=\ZZ\Gamma$, and let $P_*$ be a free resolution of $\ZZ$ whose terms $P_0,\dots,P_m$ have finite rank~\cite[VIII.4.3]{Bro}. Let $Z_m=\ker(d_m\colon P_m\to P_{m-1})$, let $I$ be its underlying set, and put $N=R^I$. Since $P_m$ and $P_{m-1}$ have finite rank, $N\otimes_RP_q\cong P_q^I$ for $q=m-1,m$, compatibly with $1\otimes d_m$. The family $\xi=(z)_{z\in Z_m}$ is therefore a cycle in $N\otimes_RP_m$, and $H_m(\Gamma,N)=0$ gives $\xi=(1\otimes d_{m+1})(y)$ for some $y=\sum_{a=1}^sx_a\otimes p_a$ with $x_a\in N$ and $p_a\in P_{m+1}$. In the coordinate indexed by $z$, this reads $z=\sum_a(x_a)_z\,d_{m+1}p_a$, so $Z_m$ is generated by $d_{m+1}p_1,\dots,d_{m+1}p_s$. Replacing $P_{m+1}$ by a free module of rank $s$ mapping onto $Z_m$, and the higher terms by a free resolution of its kernel, yields a free resolution that agrees with $P_*$ in degrees $\le m$ and has finite rank in degree $m+1$. Iterating produces a single free resolution of finite rank in every degree.
\end{proof}

\begin{proof}[Proof of Proposition~\ref{app:torus}]
By Corollary~\ref{cor:raising} with $m=1$, $E$ is of type $\FP{2}$. By Lemma~\ref{app:product-vanishing}, $H_q(E,(\ZZ E)^I)=0$ for every set $I$ and every $q\ge2$, and Lemma~\ref{app:syzygies} with $m=2$ concludes.
\end{proof}

\begin{remark}\label{app:degree-one}
Corollary~\ref{cor:raising} can be avoided. Since $B$ is finitely generated, $H_0(B,(\ZZ E)^I)\cong\ZZ[E/B]^I$, with $T$ acting coordinatewise by $gB\mapsto gtB$, so $1-T$ is injective there by~\eqref{eq:coset}. With Lemmas~\ref{app:hnn-sequence} and~\ref{app:product-vanishing} this gives $H_1(E,(\ZZ E)^I)=0$, and Lemma~\ref{app:syzygies} applies with $m=1$.
\end{remark}

\subsection{Proof of Theorem~\ref{app:main}}
Let $B=B_{1,\orc}$ be the group of Lemma~\ref{lem:initial}. By Lemma~\ref{lem:presentations}, $K_\Tr(B)$ is finitely generated and $\orc$-recursively presented, so there is an embedding $\kappa\colon K_\Tr(B)\to B$. The endomorphisms $\alpha_v=\kappa\circ\iota_v$ satisfy the hypothesis of Corollary~\ref{cor:row} by~\eqref{eq:KSigma}, and $\alpha_0$ is injective by Lemma~\ref{lem:injective}. By Proposition~\ref{app:torus}, the group
\[
E_\orc=\langle B,t\colon t^{-1}bt=\alpha_0(b),\ b\in B\rangle
\]
is of type $\FP{\infty}$. It contains $B$, hence a copy of every countable $\orc$-recursively presented group by Lemmas~\ref{lem:HNN} and~\ref{lem:initial}, and it is $\orc$-recursively presented by Subsection~\ref{sec:oracle}.

Let $J$ be the two-generator group obtained from $E_\orc$ and a fixed finite generating set by the construction in the proof of Lemma~\ref{lem:two-generators}. This construction does not depend on $n$, and $E_\orc$ is of type $\FP{n}$ for every $n$, so $J$ is of type $\FP{n}$ for every $n$, that is, of type $\FP{\infty}$. It contains $E_\orc$, and it is $\orc$-recursively presented, since it is obtained by adding finitely many generators and relators. This proves the first assertion. The second follows, because every countable group is $\orc$-recursively presented for a suitable oracle (Subsection~\ref{sec:oracle}).\qed

\begin{remark}\label{app:B2}
Whatever the choice of $\kappa$, the group $E_\orc$ is the group $B_{2,\orc}$ of Theorem~\ref{thm:oracle}. Thus the first ascending HNN extension in that proof is already of type $\FP{\infty}$: the finiteness degree enters only through the power of $\varphi$ and the bounds in Proposition~\ref{app:contraction}, not through the group. The conclusion $\FP{\infty}$ is optimal: by Corollary~\ref{cor:C}, a finitely generated group that is not recursively presented embeds in no group of type $\Ft{2}$.
\end{remark}


\end{document}